\documentclass[12pt,a4paper,oneside]{amsart}
\usepackage[utf8]{inputenc}
\usepackage{amsmath,amscd,hyperref, amsfonts, amssymb, amsthm,soul}
\usepackage{enumitem}
\usepackage{setspace,kantlipsum}
\usepackage{tikz-cd}
\usepackage{mathrsfs}
\usepackage{textcomp}
\usepackage[margin=1in]{geometry}
\usepackage{colonequals}
\usepackage{mathtools}

\theoremstyle{plain} 
\newtheorem{thm}{Theorem}[section]
\newtheorem{lemma}[thm]{Lemma}
\newtheorem{proposition}[thm]{Proposition}
\newtheorem{corollary}[thm]{Corollary}

\theoremstyle{definition}

\newtheorem{definition}[thm]{Definition}

\newtheorem{remark}[thm]{Remark}

\newtheorem{convention}[thm]{Convention}

\newcommand{\C}{\mathbb{C}}
\newcommand{\Z}{\mathbb{Z}}
\newcommand{\Q}{\mathbb{Q}}

\newcommand{\N}{\mathbb{N}}

\newcommand{\D}{\mathbb{D}}

\newcommand{\cB}{\mathcal{B}}

\newcommand{\cF}{\mathcal{F}}

\newcommand{\cI}{\mathcal{I}}
\newcommand{\cJ}{\mathcal{J}}

\newcommand{\cM}{\mathcal{M}}

\newcommand{\cO}{\mathcal{O}}

\newcommand{\cT}{\mathcal{T}}

\newcommand{\bR}{\mathbf{R}}

\newcommand{\sD}{\mathscr{D}}

\newcommand{\sT}{\mathscr{T}}

\renewcommand{\d}{\partial}

\newcommand{\gr}{\mathrm{gr}}

\newcommand{\DR}{\mathrm{DR}}
\newcommand{\Sing}{\mathrm{Sing}}

\newcommand{\mhm}{\mathrm{MHM}}
\newcommand{\td}{\mathrm{td}}

\newcounter{tmp}
\begin{document}

\title[Higher Du Bois and higher rational singularities]{Higher du Bois and higher rational singularities of hypersurfaces}

\author{Lauren\c{t}iu G. Maxim}
\address{Department of Mathematics, University of Wisconsin-Madison, 480 Lincoln Drive, Madison, WI 53706-1388, USA}
\email{maxim@math.wisc.edu}

\author{Ruijie Yang}
\address{Max-Planck-Institut f\"ur Mathematik, Vivatsgasse 7, 53111 Bonn, Germany}
%\curraddr{}
\email{ruijie.yang@hu-berlin.de, ryang@mpim-bonn.mpg.de }

\begin{abstract}
In this note, we give several equivalent characterizations of higher Du Bois and higher rational singularities in the context of globally defined hypersurfaces.  As a key input, we characterize these singularities using the Hodge filtration on the vanishing cycle mixed Hodge module. As an application, we indicate a homological criterion for detecting higher Du Bois and higher rational hypersurface singularities, formulated in terms of the spectral Hirzebruch-Milnor characteristic classes.
\end{abstract}

%\begin{center}
    \date{\today}
%\end{center}



\maketitle

\section{Introduction}

The notions of higher Du Bois and higher rational singularities of hypersurfaces were recently introduced and studied in the works of Jung-Kim-Saito-Yoon \cite{JKSY}, Musta\c{t}\u{a}-Olano-Popa-Witaszek \cite{MOPW} and Friedman-Laza \cite{FL22}, as natural generalizations of Du Bois and rational singularities, respectively. In this note, we give several equivalent characterizations of higher Du Bois and higher rational singularities in the context of globally defined hypersurfaces. As a byproduct, we obtain a homological criterion for detecting higher Du Bois and higher rational singularities, formulated in terms of the spectral Hirzebruch classes of \cite{MSS}. 

A complex algebraic variety $Z$ is said to have \emph{$k$-Du Bois singularities} if, for all integers $0 \leq  i \leq  k$, the natural morphisms
\[ \Omega^i_Z \to \underline{\Omega}^i_Z \]
are isomorphisms, where $\Omega^i_Z$ is the sheaf $\wedge^i\Omega^1_Z$ of K\"ahler differentials on $Z$ and  $\underline{\Omega}^i_Z$ is the $i$-th graded piece of the Du Bois complex of $Z$ with respect to the Hodge filtration. If $k=0$, one recovers the usual notion of Du Bois singularities. On the other hand, a variety $Z$ has \emph{$k$-rational singularities} if, for any resolution of singularities $ \mu: \widetilde{Z} \to Z$ that is an isomorphism over the smooth locus of $Z$ and such that the reduced inverse image $D$ of the singular locus of $Z$ is a simple normal crossing divisor, the natural morphisms
\[ \Omega^i_Z \to  \bR\mu_{\ast}\Omega^i_{\widetilde{Z}}(\log D) \]
are isomorphisms for all $0 \leq i \leq k$. If $k=0$, this corresponds to rational singularities. It was shown in \cite{MP22,FL22a} that if $Z$ is an algebraic variety which is a local complete intersection with  $k$-rational singularities, then $Z$ has $k$-Du Bois singularities. 

Throughout this paper, we consider the following setup: $X$ is a smooth complex algebraic variety and $f\colon X\to \C$ is a non-constant holomorphic function such that $Z\colonequals f^{-1}(0)$ is reduced. Denote by 
\[\varphi_f^H:\mhm(X) \to \mhm(Z)\]
the vanishing cycle functor for mixed Hodge modules, which corresponds to the perverse vanishing cycle functor on the underlying constructible complexes. Moreover, the monodromy action on vanishing cycles gives rise to a decomposition 
\[ 
\varphi_f^H=\varphi_{f,1}^H \oplus  \varphi_{f,\neq 1}^H
\]
into unipotent and non-unipotent parts.

Let $\Q^H_X$ be the constant mixed Hodge complex on $X$, whose underlying constructible complex is the constant sheaf $\Q_X$. Since $X$ is smooth, $\Q^H_X$ is just a shifted mixed Hodge module, i.e., $\Q^H_X[\dim(X)]\in \mhm(X)$. 
It was shown in \cite{MSS} that the property that $Z=f^{-1}(0)$ has rational, resp.,  Du Bois singularities can be characterized by the vanishing of certain graded pieces of the Hodge filtration on the mixed Hodge module $\varphi^H_f\Q^H_X[\dim(X)] \in \mhm(Z)$,  resp., of its non-unipotent part. Our first result generalizes such a characterization to the case of higher rational and higher Du Bois singularities. More precisely, we show the following.
\begingroup
\setcounter{tmp}{\value{thm}}
\setcounter{thm}{0} 
\renewcommand\thethm{\Alph{thm}}
\begin{thm}\label{thm: main 4}
Let $X$ be a smooth complex algebraic variety and let $f\colon X \to \C$ be a non-constant holomorphic function such that $Z=f^{-1}(0)$ is reduced. Then, for any $k\in \N$,  
$Z$ has $k$-Du Bois singularities 
if and only if
\[\begin{cases}
\gr^F_p\varphi^H_{f,\neq 1}\Q^H_X[\dim X]=0, &\quad p\leq k+1,\\
\gr^F_{p}\varphi^H_{f,1}\Q^H_X[\dim X]=0, &\quad  p\leq k.
\end{cases} \]
Similarly, $Z$ has $k$-rational singularities
if and only if
\[ \gr^F_p\varphi^H_{f}\Q^H_X[\dim X]=0, \quad p\leq k+1.\]
\end{thm}
\endgroup

When $k=0$, this recovers  \cite[(4.2.4) and (4.3.8)]{MSS}, up to a left-to-right transformation of $\sD$-modules. A cohomological version of Theorem \ref{thm: main 4} was established in \cite[Theorem 7.6]{KL20}, where the authors considered a one-dimensional projective family such that the central fiber $Z$ is reduced, and proved that $Z$ being $k$-Du Bois implies that the Hodge filtrations on the cohomology of $Z$ and, resp., cohomology of generic fiber, agree up to the $k$-th level. When $f$ is a proper map, such a result can be seen as a direct consequence of Theorem \ref{thm: main 4}.

We next recall that for each integer $k\geq 0$ and rational number $\alpha$ one can define as in \cite{SchnellYang}
ideal sheaves $\cI_{k,\alpha}(Z)$ and $\cI_{k,<\alpha}(Z)$ of $\cO_X$ associated to the hypersurface $Z=f^{-1}(0)$. 
We consider here the \emph{$k$-th multiplier ideal} of $Z$ by
\[ \cJ_k(\alpha Z)\colonequals \cI_{k,<-\alpha}(Z), \quad k\in \N, \ \alpha \in \Q.\]
The results of \cite{SchnellYang} show that for $\alpha>0$ we have $\cJ_0(\alpha Z)=\cJ(\alpha Z)$, the usual multiplier ideal. Moreover, $\cJ_{k}(\alpha Z)$ is decreasing and right-continuous in $\alpha$, and $\cJ_{k}(-kZ)=\cO_X$.  We say that $\alpha$ is a \emph{jumping number of the $k$-th multiplier ideal $\cJ_k$ of $Z$} if 
\[\cJ_{k}((\alpha-\epsilon)Z)/\cJ_{k}(\alpha Z)\neq 0, \quad \textrm{for  $0<\epsilon\ll 1$}.\]
Our next result relates jumping numbers of higher multiplier ideals with higher Du Bois and higher rational singularities.
\begingroup
\setcounter{tmp}{\value{thm}}
\setcounter{thm}{1} 
\renewcommand\thethm{\Alph{thm}}
\begin{thm}\label{prop: higher du Bois and higher rational via higher multiplier ideals}
With the above notations, $Z$ has $k$-Du Bois singularities if and only if
\[ \cJ_{\ell}((\alpha-\epsilon)Z)/\cJ_{\ell}(\alpha Z)=0, \quad \textrm{for all} \quad 0\leq \ell\leq k, \ 0\leq \alpha <1,\]
or, equivalently, there are no jumping numbers in $[0,1)\cap \Q$ of the $\ell$-th multiplier ideal $\cJ_{\ell}$ for all $0\leq \ell \leq k$. Furthermore, $Z$ has $k$-rational singularities if and only if
\[ \cJ_{\ell}((\alpha-\epsilon)Z)/\cJ_{\ell}(\alpha Z)=0, \quad \textrm{for all \ $0\leq \ell\leq k$, $0\leq \alpha <1$}, \textrm{ and $\ell=k+1$, $\alpha=0$}.\]
\end{thm}
\endgroup

In the context of globally defined hypersurfaces, Theorem \ref{thm: main 4} and Theorem \ref{prop: higher du Bois and higher rational via higher multiplier ideals} provide a vanishing cycle, resp., higher multiplier ideal interpretation of the results by Musta\c{t}\u{a}-Popa \cite[Theorem B]{MP22} and Friedman-Laza \cite[Theorem 1.6]{FL22a} showing that $k$-rational implies $k$-Du Bois and, respectively,  Musta\c{t}\u{a}-Popa \cite[Corollary F]{MP22}, according to which $k$-Du Bois implies $(k-1)$-rational. A key input for Theorem \ref{prop: higher du Bois and higher rational via higher multiplier ideals} is a new characterization of Saito's minimal exponent, given in Proposition \ref{prop: minimal exponent via jumping numbers of higher multiplier ideals}. This makes the minimal exponent computable by birational geometry techniques.

A further refinement of Theorem \ref{thm: main 4} is provided by the following result, which connects jumping numbers of higher multiplier ideals to the Hodge filtration of the corresponding eigenspace of the vanishing cycle mixed Hodge module.
\begingroup
\setcounter{tmp}{\value{thm}}
\setcounter{thm}{2} 
\renewcommand\thethm{\Alph{thm}}
\begin{proposition}\label{thm: main 5}
With the above notations, $\alpha \in (0,1)\cap \Q$ is not a jumping number of the $\ell$-th multiplier ideal $\cJ_\ell$ of $Z$ for all $\ell\leq k$ if, and only if, 
\[\gr^F_\ell \varphi^H_{f,e^{-2\pi i\alpha}}\Q^H_X[\dim X] =0,\quad \textrm{for all \ $\ell \leq k+1$}.\]
Moreover, $\alpha=0$ is not a jumping number of the $\ell$-th multiplier ideal $\cJ_\ell$ of $Z$ for all $\ell\leq k$ if, and only if, 
\[\gr^F_\ell \varphi^H_{f,1}\Q^H_X[\dim X] =0,\quad \textrm{for all \ $\ell \leq k$}.\]
\end{proposition}
\endgroup
%}

Using Theorem \ref{thm: main 4} and the Thom-Sebastiani formula for vanishing cycles proved in \cite[Theorem 2]{MSS20}, 
we also provide here a corresponding Thom-Sebastiani type statement for higher rational and higher Du Bois singularities. See Proposition \ref{TS} for the exact formulation.

\medskip

At the end of this introduction, we discuss a homological criterion for detecting higher Du Bois and, resp, higher rational singularities of globally defined hypersurfaces, which we formulate in terms of the spectral Hirzebruch-Milnor classes introduced in \cite{MSS}. Before stating our results, let us first recall from \cite{BSY,S,MS} that, for any bounded complex of mixed Hodge modules $M^\bullet \in D^b\mhm(Z)$ on $Z$, there is an associated \emph{homology Hirzebruch class} 
\[ T_{y\ast}(M^\bullet) \in H_*(Z)[y,y^{-1}],\]
where $H_{\ast}(Z)=\oplus_k H_k(Z)$ and $H_k(Z)$ denotes either the Borel-Moore homology group $H^{BM}_{2k}(Z,\mathbb{Q})$ or the Chow group $CH_k(Z)_\mathbb{Q}$. This class is defined as the image of a natural transformation on the Grothendieck group $K_0(D^b\mhm(Z))$, see \S \ref{Hirzebruch classes of mixed Hodge modules and Hirzebruch-Milnor classes of hypersurfaces} for more details. When $Z\colonequals f^{-1}(0)$ is a reduced hypersurface defined by a non-constant holomorphic function $f\colon X\to \C$ on a smooth complex algebraic variety $X$, one can then consider the {\it Hirzebruch-Milnor class of $Z$} given by (cf. \cite{CMSS, MSS1})
\[ {M_{y\ast}(Z):=T_{y\ast}(\varphi_f\Q^H_X)} \in H_*(\Sing(Z))[y],\]
with $\varphi_f\colonequals \varphi_f^H[1]$ corresponding to the usual Deligne vanishing cycle functor on constructible complexes. By taking into account the monodromy action on vanishing cycles, one can further define as in \cite{MSS} a {\it spectral Hirzebruch-Milnor class} $$\cM^{sp}_{t\ast}(Z) \in H_*(\Sing(Z))[t^{1/ord(T_s)}]$$ of $Z$, with $T_s$ denoting the semi-simple part of the monodromy; see \S \ref{sec: spectral classes} for a definition.

As shown in \cite{MSS}, the spectral Hirzebruch-Milnor class $\cM^{sp}_{t\ast}(Z)$ of $Z=f^{-1}(0)$ is sensitive to $Z$ having rational or Du Bois singularities. In this note, we generalize the results of \cite{MSS} as follows. For $\alpha \in \Q$, we denote the coefficient of $t^\alpha$ in the spectral Hirzebruch-Milnor class $\cM^{sp}_{t\ast}(Z)$ of $Z$ by 
\[ \cM^{sp}_{t\ast}(Z)|_{t^{\alpha}}\in H_{\ast}(\Sing(Z)).\] Then we prove the following.
\begingroup
\setcounter{tmp}{\value{thm}}
\setcounter{thm}{3} 
\renewcommand\thethm{\Alph{thm}}
\begin{thm}\label{dbrsp}
With the above notations, if $Z$ has $k$-Du Bois singularities then $\cM^{sp}_{t*}(Z)|_{t^\alpha} = 0$ for all $\alpha < k+1$, 
with the converse implication being true if $\Sing(Z)$ is projective.

Moreover, if $Z$ has $k$-rational singularities then $\cM^{sp}_{t*}(Z)|_{t^\alpha} = 0$ for all $\alpha \leq k+1$, 
with the converse implication being true if $\Sing(Z)$ is projective.
\end{thm}
\endgroup

For some examples when the assumption that $\Sing(Z)$ is projective is satisfied, see \cite{Saito23}.

Moreover, as shown in \cite[Proposition 2]{MSS}, the jumping coefficients of the multiplier ideal of the hypersurface $Z$ can also be detected by suitable components of the spectral Hirzebruch-Milnor class of $Z$. A generalization of this fact can be deduced by combining Theorems \ref{prop: higher du Bois and higher rational via higher multiplier ideals} and \ref{dbrsp} as follows.
\begingroup
\setcounter{tmp}{\value{thm}}
\setcounter{thm}{4} 
\renewcommand\thethm{\Alph{thm}}
\begin{corollary}\label{thm: main 3}
With the above notations, if $\alpha\in [0,1)\cap \Q$ is not a jumping number of the $\ell$-th multiplier ideal $
\cJ_{\ell}$ of $Z$ for all $0\leq \ell\leq k$, then 
\[ \cM^{sp}_{t\ast}(Z)|_{t^{\ell+\alpha}}=0, \quad \textrm{for all \ $0\leq \ell \leq k$}.\]
The converse holds if $\Sing(Z)$ is projective.
\end{corollary}  
\endgroup

Let us note that, if for an integer $k \geq 0$ we denote by  
\[
[M_{y\ast }(Z)]_{k}  \in H_\ast(\Sing(Z))
\]
the coefficient of $y^k$ in $M_{y\ast}(Z)$, then we have the identity
\begin{equation}\label{MHsp} [M_{(-y)\ast}(Z)]_{k}=\bigoplus_{\alpha \in \mathbb{Q}\cap[0,1)} \cM^{sp}_{t\ast}.(Z)|_{t^{k+\alpha}},\end{equation}
In particular, as a consequence of Theorem \ref{dbrsp}, we get the following interpretation of $Z$ being $k$-Du Bois without involving the monodromy action (generalizing the case $k=0$ considered in \cite[Theorem 5]{MSS}).
\begingroup
\setcounter{tmp}{\value{thm}}
\setcounter{thm}{5} 
\renewcommand\thethm{\Alph{thm}}
\begin{corollary}\label{thm: main 1}
With the above notations, if $Z$ has $k$-Du Bois singularities, then
\[[M_{y\ast}(Z)]_{p}=0 \in H_{\ast}(\Sing(Z)) \textrm{ for all \ $p \leq k$}.\]
If $\Sing(Z)$ is projective, then the converse is true.
\end{corollary}
\endgroup

Similarly, using the decomposition $\varphi_f=\varphi_{f,1}\oplus \varphi_{f,\neq 1}$, with $\varphi_f\colonequals \varphi_f^H[1]$ as before, we also define the {\it unipotent Hirzebruch-Milnor class of $Z$} by
\begin{equation}\label{unih} M^u_{y\ast}(Z):=T_{y\ast}(\varphi_{f,1}\Q^H_X) \in H_\ast(\Sing(Z))[y],\end{equation}
and we denote by $[M^u_{y\ast}(Z)]_{k+1}$  the coefficient of $y^{k+1}$ in \eqref{unih}. Then Theorem \ref{dbrsp} implies the following generalization of \cite[(4.2.6)]{MSS} where the case $k=0$ was considered. See also Remark \ref{remnew} for an interpretation of $k$-rationality which does not involve the monodromy action.
\begingroup
\setcounter{tmp}{\value{thm}}
\setcounter{thm}{6} 
\renewcommand\thethm{\Alph{thm}}
\begin{corollary}\label{thm: main 2}
With the above notations, if $Z$ has $k$-rational singularities, then
\begin{center} 
$[M_{y\ast}(Z)]_{p}=0 \in H_{\ast}(\Sing(Z))$ for all \ $p \leq k$,  and  $[M^u_{y\ast}(Z)]_{k+1}=0  \in H_{\ast}(\Sing(Z))$.
\end{center}
If $\Sing(Z)$ is projective, then the converse is true.
\end{corollary}
\endgroup

\begin{remark}
    (i) Even though we refer on several occasions to the upcoming paper \cite{SchnellYang}, all results used here are explicitly proved in this paper. \newline
    (ii) Since the first version of this paper became public, some of its results have been generalized to the case of arbitrary hypersurfaces in smooth complex varieties, see \cite{MSY23}. Therefore the current paper is superseded by \cite{MSY23}.\newline
    (iii) For results in \S 2, some simplified arguments are given in \cite{Saito23}.
\end{remark}

\begin{convention}
Throughout the paper, we work with \emph{left} $\sD$-modules and \emph{increasing} Hodge filtrations.
\end{convention}

\subsection*{Acknowledgement}
We would like to thank Bradley Dirks, Radu Laza, Mircea Musta\c{t}\u{a}, Mihnea Popa, Christian Schnell and J\"org Sch\"urmann for helpful discussions. We also thank Morihiko Saito for pointing out an inconsistency in our earlier convention on the indexing of Hodge filtration on vanishing cycles. During the preparation of this paper, both authors visited the Max Planck Institute for Mathematics in Bonn, and we would like to thank the institute for hospitality, excellent working conditions and  financial support.
The first named author is partially supported by the Simons Foundation (Collaboration Grant \#567077), and by the Romanian Ministry of National Education (CNCS-UEFISCDI grant PN-III-P4-ID-PCE-2020-0029).

%%%%%%%%%%%%%%%%%%%%%%%%%%%%%%%%%%%%%%%%%%%%%%

\section{Higher Du Bois and Higher rational singularities}\label{sec: higher Du Bois and higher rational}

In this section, we review the definition of higher Du Bois and higher rational singularities, and prove Theorem \ref{thm: main 4}.

%Proposition \ref{prop: higher du Bois and higher rational via higher multiplier ideals} , Theorem \ref{thm: main 4} and Proposition \ref{thm: main 5}.

\subsection{Definitions}
Let $Z$ be a complex algebraic variety. There are two natural generalization of the De Rham complex of smooth varieties. First, we denote by $\Omega^1_Z$ the sheaf of K\"ahler differentials on $Z$ and let $\Omega^i_Z\colonequals\wedge^i \Omega^1_Z$. We still call the resulting filtered complex $(\Omega^{\bullet}_Z,F)$ the De Rham complex with the Hodge filtration. Another generalization is provided by \cite{duBois}, where one associates 
canonically and functorially to $Z$ its Du Bois
complex $(\underline{\Omega}^{\bullet}_Z,F)$ with a filtration. There is a natural morphism of filtered complexes
\begin{equation}\label{eqn: map from de Rham to Du Bois}
(\Omega^{\bullet}_Z,F) \to (\underline{\Omega}^{\bullet}_Z,F). 
\end{equation}
If $Z$ is smooth, then \eqref{eqn: map from de Rham to Du Bois} is a filtered quasi-isomorphism. For an introduction and further references, see \cite{Kovacs}.

For each $i \geq 0$, consider the shifted graded piece
\[ \underline{\Omega}^i_Z\colonequals \gr^i_F\underline{\Omega}^{\bullet}_Z[i] \in D^b_{\mathrm{coh}}(Z),\]
where $D^b_{\mathrm{coh}}(Z)$ denotes the bounded derived category of sheaves of  $\mathcal{O}_Z$-modules with coherent cohomologies. %coherent sheaves on $Z$. 
The morphism \eqref{eqn: map from de Rham to Du Bois} induces a morphism in $D^b_{\mathrm{coh}}(Z)$
\begin{equation}\label{eqn: graded piece of map from de Rham to Du Bois}
 \Omega^i_Z \to \underline{\Omega}^i_Z.
\end{equation}
The following definition was recently proposed by Jung-Kim-Saito-Yoon \cite{JKSY} and also studied by Musta\c{t}\u{a}-Olano-Popa-Witaszek \cite{MOPW} and Friedman-Laza \cite{FL22,FL22a}.

\begin{definition}
For an integer $k\geq 0$, we say that $Z$ has {\it $k$-Du Bois singularities} if the morphism \eqref{eqn: graded piece of map from de Rham to Du Bois} is an isomorphism  for all $0\leq i \leq k$. 
\end{definition}
\begin{remark}
When $k=0$, this recovers the usual notion of Du Bois singularities. 
\end{remark}

Let $Z$ be an irreducible variety and let $ \mu: \widetilde{Z} \to Z$ be a resolution of singularities that is an isomorphism over the smooth locus of $Z$ and such that the reduced inverse image $D$ of the singular locus  of $Z$ is a simple normal crossing divisor. For each $i$, there is a natural morphism
\begin{equation}\label{eqn: morphism from Du Bois complex to log complex}
\Omega^i_Z \to  \bR\mu_{\ast}\Omega^i_{\widetilde{Z}}(\log D). 
\end{equation}
The following definition is due to Friedman-Laza, see \cite[Definition 2.1]{MP22}.
\begin{definition}
For an integer $k\geq 0$, we say that an irreducible variety $Z$ has {\it $k$-rational singularities} if the morphism \eqref{eqn: morphism from Du Bois complex to log complex} is an isomorphism for all $0\leq i \leq k$. We say that an arbitrary variety $Z$ has $k$-rational
singularities if all its connected components are irreducible, with $k$-rational singularities.
\end{definition}
\begin{remark}
When $k=0$, this recovers the usual notion of rational singularities. As in \cite[Remark 2.3]{MP22}, one can show that this is independent of the choice of the resolution.
\end{remark}
\begin{remark}
Higher rational singularities were initially defined by using the Grothendieck duality functor on the Du Bois complex, see  \cite[Definition 3.12]{FL22a}. It was then shown that this is equivalent to the definition above, see \cite[Corollary 3.17]{FL22a}.
\end{remark}

%%%%%%%%%%%%%%%%%

\subsection{Characterization via minimal exponents}\label{sec: k-Du Bois singularity via minimal exponents}

In this section, we review several results on the characterization of higher Du Bois and higher rational singularities using minimal exponents.

Let $X$ be a smooth complex algebraic variety and let $f$ be a holomorphic function on $X$ such that $Z\colonequals f^{-1}(0)$ is reduced. The \emph{minimal exponent} $\tilde{\alpha}_{f}$ is defined to be the smallest root of $b_f(-s)/(-s+1)$, where $b_f(s)$ is the Bernstein-Sato polynomial associated to $f$. This is well-defined because $1$ is a jumping coefficient of $Z=f^{-1}(0)$, hence a root of $b_f(-s)$. 
%If $D$ is a divisor on $X$, for any point $x\in D$, let $f_x$ be the local defining equation of $D$ such that $f_x(x)=0$. Then the minimal exponent $\tilde{\alpha}_{D,x}$ is defined by 
%\[ \tilde{\alpha}_{D,x}\colonequals \tilde{\alpha}_{f_x}.\]
%This is independent of the choice of $f_x$ because the Bernstein-Sato polynomial is so.
%The minimal exponent of $D$ is defined to be
%\[ \tilde{\alpha}_{D}\colonequals  \min_{x\in D} \{\tilde{\alpha}_{D,x}\}.\]

We collect below results from \cite[Theorem 1]{JKSY}, \cite[Theorem 1.1]{MOPW} and \cite[Theoerm E]{MP22}.
\begin{thm}\label{thm: k du Bois singularity function via minimal exponent}
With the above notations, 
\begin{itemize}

\item $Z$ has $k$-Du Bois singularities if and only if $\tilde{\alpha}_{f} \geq  k + 1$.
\item $Z$ has $k$-rational singularities if and only if $\tilde{\alpha}_{f} >  k + 1$.
\end{itemize}
%Moreover, let $D$ be a divisor on $X$, then 
%  \[  \textrm{$D$ has $k$-Du Bois singularities in a neighborhood of $x$} \Longleftrightarrow 
%\tilde{\alpha}_{D,x} \geq  k + 1.\]
\end{thm}
It follows that for such hypersurface singularities, we have 
\[ \textrm{$k$-rational}\Longrightarrow \textrm{ $k$-Du Bois} \Longrightarrow \textrm{$(k-1)$-rational}, \]
which generalizes Kov\'acs' result that rational singularities are Du Bois \cite{Kovacs99}.
%\begin{proof}
%This is because $D$ has $k$-du Bois singularity if and only if $D$ has $k$-du Bois singularity at every point $x\in X$.
%\end{proof}

\subsection{Microlocal $V$-filtration}

In this section, we recall Saito's microlocal $V$-filtration \cite{Saito16}, which encodes the minimal exponents. We then recall a result of Schnell and the second author about relating the microlocal $V$-filtration with the Kashiwara-Malgrange $V$-filtration, which we use to obtain a characterization of higher Du Bois and higher rational singularities in terms of the Hodge filtration on the associated graded of $V$-filtrations.

Let $X$ be a smooth complex algebraic variety and let $f$ be a holomorphic function on $X$ such that $Z=f^{-1}(0)$ is reduced. In \cite{Saito94}, Saito introduced the microlocal $V$-filtration on $\cO_X$, denoted by $\tilde{V}^{\alpha}\cO_X$. For the benefit of the reader, we recall its construction below. First, we fix some conventions.  Denote the graph embedding by 
 \[i_f:X\to X\times \C_t, \ x\mapsto (x,f(x)),\]
 where $t$ is the coordinate on $\C$.  Denote by $(i_f)_{+}$  the direct image functor for filtered $\sD$-modules.
\begin{convention}\label{convention for V filtration}
 On a  \emph{left} $\sD_X$-module $\cM$ on $X$ %with an increasing filtration $F$ 
one has the $V$-filtration associated to $f$, which is decreasing and satisfies the property that $\d_t t-\alpha$ is nilpotent on $\gr^{\alpha}_V\cM_f$, where 
 \[\cM_f\colonequals (i_f)_{+}\cM.\]
If $M$ is a  mixed Hodge module with underlying filtered $\sD$-module $(\cM,F)$ with increasing filtration $F$, the filtered $\sD$-module underlying the vanishing cycle mixed Hodge module $\varphi^H_fM$ with the decomposition 
 \[ \varphi^H_fM = \varphi^H_{f,\neq 1}M \oplus \varphi^H_{f,1}M\]
 is 
 \begin{align*}
 \varphi^H_f\cM &=\bigoplus_{0\leq \alpha< 1}\gr^{\alpha}_V\cM_f,\\
 \varphi^H_{f,\neq 1}\cM&=\bigoplus_{0<\alpha <1} \gr^{\alpha}_V\cM_f, \quad  \varphi^H_{f,1}\cM=\gr^{0}_V\cM_f,
 \end{align*}
 with 
 \begin{align*}
 F_k \varphi^H_{f,\neq 1}\cM&=\bigoplus_{0<\alpha <1} F_{k-1}\gr^{\alpha}_V\cM_f, \quad F_k \varphi^H_{f,1}\cM=F_{k}\gr^{0}_V\cM_f.
 \end{align*}
\end{convention}
%{\color{red} Here do we mean $\varphi_fM$ or $\varphi^H_fM$? Note that $\varphi_fM$ is not a mixed Hodge module, but $\varphi^H_fM$ is.}
 %{\color{purple} We mean $\varphi^H_fM$. I have changed through.}

For the left filtered $\sD_X$-module $(\cO_X,F)$ underlying the constant mixed Hodge module $\Q^H_X[\dim X]$, there are two associated filtered left $\sD$-modules 
\[ \cB_f\colonequals (i_f)_{+}\cO_X, \quad \tilde{\cB}_f \colonequals \cO_X\otimes_{\C}\C[\d_t,\d_t^{-1}],\]
where $\cB_f\cong \cO_X\otimes_{\C}\C[\d_t]$, and $\tilde{\cB}_f$ is the so-called \emph{algebraic partial microlocalization} of ${\cB}_f$. 
The Hodge filtrations on $\cB_f$ and $\tilde{\cB}_f$ are defined by
\[ F_k\cB_f\colonequals \sum_{0\leq \ell\leq k} \cO_X\otimes \d_t^{\ell},  \quad  F_k\tilde{\cB}_f\colonequals \sum_{\ell\leq k} \cO_X\otimes \d_t^{\ell},\]
so that
\begin{equation}\label{eqn: associated graded of Hodge filtrations on Bf}
\gr^F_k\cB_f=\cO_X\otimes \d_t^k, \quad \forall k\in \N, \quad {\rm and} \quad \gr^F_k\tilde{\cB}_f=\cO_X\otimes \d_t^k, \quad \forall k\in \Z.
\end{equation}
In particular, there is a natural isomorphism
\begin{equation}\label{eqn: OX embeds into the graph embedding}
    \cO_X\otimes 1\cong  \gr^F_0\tilde{\cB}_f.
\end{equation} 
The \emph{microlocal $V$-filtration} on $\tilde{\cB}_f$ along $t=0$ is defined by 
\begin{equation*} 
V^{\beta}\tilde{\cB}_f\colonequals \left\{
\begin{array}{rl}
V^{\beta}\cB_f\oplus (\cO_X[\d_t^{-1}]\d_t^{-1}) & \text{if } \beta \leq 1, \\ 
\d_t^{-j}\cdot V^{\beta-j}\tilde{\cB}_f & \text{if } \beta>1, \beta-j \in (0,1].
\end{array} \right. 
\end{equation*}
\begin{remark}\label{remark: convention on Hodge filtration on Bf}
Here we use the same convention for the Hodge filtration as in \cite[(1.1.3) and (1.3.1)]{Saito16} and \cite[(1.2.4)]{Saito94}, but different from the ones in \cite[(3.1.2) and (3.1.3)]{MSS}, where one indexes $F$ as in the case of right $\sD$-modules.
\end{remark}

\begin{definition}\label{definition: microlocal V filtration}
For $\beta\in \Q$, the \emph{microlocal $V$-filtration} $\tilde{V}^{\beta}\cO_X$ on $\cO_X$ induced by $f$ is defined using the isomorphism \eqref{eqn: OX embeds into the graph embedding}
\[ \tilde{V}^{\beta}\cO_X\otimes 1 \colonequals \gr^F_0V^{\beta}\tilde{\cB}_f.\]
\end{definition}
\begin{remark}\label{remark: microlocal V filtration is decreasing}
In \cite{Saito94}, Saito proved that $V^{\bullet}\tilde{\cB}_f$ is a decreasing filtration and so is $\tilde{V}^{\bullet}\cO_X$.
\end{remark}

\begin{remark}\label{remark: V>beta}
Since the filtration  $V^{\bullet}\cB_f$ is discrete, we have a well-defined notation
\[ V^{>\beta}\cB_f\colonequals \bigcup_{\alpha>\beta} V^{\alpha}\cB_f=V^{\beta+\epsilon}\cB_f, \quad \textrm{for some $0<\epsilon \ll 1$}.\]
Similarly, the notations $V^{>\beta}\tilde{\cB}_f$ and $\tilde{V}^{>\beta}\cO_X$ are well-defined.
\end{remark}
In \cite[(1.3.8)]{Saito16}, Saito proved that the minimal exponent $\tilde{\alpha}_f$ in \S \ref{sec: k-Du Bois singularity via minimal exponents} coincides with the minimal jumping number of $\tilde{V}^{\bullet}\cO_X$.
\begin{thm}[Saito]\label{thm: minimal exponent via microlocal V filtration}
With the above notations, we have
\[ \tilde{\alpha}_f= \max \{\gamma \mid \tilde{V}^{\gamma}\cO_X=\cO_X\}.\]
In other words, $\tilde{V}^{\gamma}\cO_X=\cO_X$ if and only if $\tilde{\alpha}_f\geq \gamma$.
\end{thm}

In \cite{SchnellYang}, Schnell and the second author showed the following result. For reader's convenience, we include the proof below.
\begin{lemma}%[Schnell-Yang]
\label{lemma: from microlocal V to V filtration}
If $0<\beta \leq 1$ and $k\in \N$, then we have an isomorphism
\[ \tilde{V}^{k+\beta}\cO_X\otimes \d_t^k\cong \gr^F_{k}V^{\beta}\cB_f, \]
under the identification $\gr^F_k\cB_f\cong \cO_X\otimes \d_t^k$.
\end{lemma}

\begin{proof}
Let $u\in \tilde{V}^{k+\beta}\cO_X\subseteq \cO_X$. Since $k\in \N$ and $\beta\in (0,1]$, by Definition \ref{definition: microlocal V filtration} there exist $u_{-\ell},u_{-\ell+1},\ldots u_0\in \cO_X$, $\ell\geq 0$, such that $u=u_0$ and
\[ \sum_{-\ell \leq i\leq 0} u_i\otimes \d^i_t \in V^{k+\beta}\tilde{\cB}_f \cap F_0\tilde{\cB}_f = \d^{-k}_t\cdot (V^{\beta}\cB_f\oplus \cO_X[\d_t^{-1}]\d_t^{-1}) \cap F_0\tilde{\cB}_f.\]
Therefore
\[ \sum_{-\ell\leq i
\leq 0}u_i\otimes \d^{k+i}_t =\d^k_t\cdot (\sum_{-\ell \leq i\leq 0} u_i\otimes \d^i_t)\in V^{\beta}\cB_f\oplus \cO_X[\d_t^{-1}]\d_t^{-1}.\]
Since $k\geq 0$ and 
\[ u\otimes \d_t^k=u_0\otimes \d^k_t,\]
we know that
\[ \sum_{i+k\geq 0} u_i\otimes \d^{i+k}_t \in V^{\beta}\cB_f\cap F_k\cB_f.\]
Hence $u\otimes \d^k_t=u_0\otimes \d^k_t$ gives an element in $\gr^F_kV^{\beta}\cB_f$ via the isomorphism $\gr^F_k\cB_f\cong \cO_X\otimes \d^k_t$. This induces a map
\begin{equation}\label{eqn: map from microlocal to V}
\tilde{V}^{k+\beta}\cO_X \otimes \d^k_t \to \gr^F_kV^{\beta}\cB_f.
\end{equation}

We next show that this map is an isomorphism. First, suppose that $u\otimes \d_t^k$ is $0$ in $\gr^F_{k}V^{\beta}\cB_f$, hence $u\otimes \d_t^k\in F_{k-1}V^{\beta}\cB_f$. We must have $u=0$ and hence \eqref{eqn: map from microlocal to V} is injective. Now, if $v\otimes \d_t^k $ represents an element in $\gr^F_kV^{\beta}\cB_f$, there exist $v_0,\ldots,v_k=v \in \cO_X$ such that
\[ \sum_{0\leq i\leq k} v_i\otimes \d^i_t\in V^{\beta}\cB_f.\]
Then
\[ \sum_{0\leq i\leq k} v_i\otimes \d^{i-k}_t \in \d_t^{-k}\cdot V^{\beta}\cB_f\cap F_0\tilde{\cB}_f\subseteq V^{k+\beta}\tilde{\cB}_f\cap F_0\tilde{\cB}_f. \]
Therefore $v\otimes 1=v_k\otimes 1$ gives an element in 
\[\tilde{V}^{k+\beta}\cO_X\otimes 1=\frac{F_0V^{k+\beta}\tilde{\cB}_f}{F_{-1}V^{k+\beta}\tilde{\cB}_f}\hookrightarrow \gr^F_0\tilde{\cB}_f\cong \cO_X\otimes 1.
\]
Therefore \eqref{eqn: map from microlocal to V} is also surjective.

\end{proof}

Lemma \ref{lemma: from microlocal V to V filtration} enables us to use the Hodge filtration on $\gr^{\bullet}_V\cB_f$ to understand the jumping behaviour of the microlocal $V$-filtration.
\begin{proposition}\label{proposition: jumping of microlocal V filtration}
With the notations above, for $\gamma\in\Q$, we have an isomorphism
\begin{equation*}
\tilde{V}^{\gamma}\cO_X/\tilde{V}^{>\gamma}\cO_X\cong  \gr^F_j\gr^{\gamma-j}_V\cB_f,
\end{equation*}
where $j=\lfloor \gamma \rfloor$ is the integer part of $\gamma$. 
\end{proposition}

\begin{proof}
Let $j$ be the unique integer such that  $j\leq \gamma <j+1$. There are two cases to consider.

Case 1: $j<\gamma<j+1$. Lemma \ref{lemma: from microlocal V to V filtration} gives
\[ \tilde{V}^{\gamma}\cO_X\otimes \d_t^j=\gr^F_{j}V^{\gamma-j}\cB_f,\quad \tilde{V}^{>\gamma}\cO_X\otimes \d_t^j=\gr^F_{j}V^{>\gamma-j}\cB_f.\]
Therefore
\[ \tilde{V}^{\gamma}\cO_X/\tilde{V}^{>\gamma}\cO_X \cong \frac{\tilde{V}^{\gamma}\cO_X\otimes \d_t^j}{\tilde{V}^{>\gamma}\cO_X\otimes \d_t^j}=\gr^F_j\gr^{\gamma-j}_V\cB_f.\]

Case 2: $\gamma=j$. By Lemma \ref{lemma: from microlocal V to V filtration} we have
\[ \tilde{V}^{j}\cO_X\otimes \d_t^{j-1}=\gr^F_{j-1}V^1\cB_f.\]
This implies that there is an equality inside $ \cO_X\otimes \d_t^j=\gr^F_{j}\cB_f$:
\[ \tilde{V}^{j}\cO_X\otimes \d_t^{j}=\d_t(\gr^F_{j-1}V^1\cB_f).\]
Lemma \ref{lemma: from microlocal V to V filtration} also shows that 
\[ \tilde{V}^{>j}\cO_X\otimes \d_t^j =\gr^F_{j}V^{>0}\cB_f.\]
Since $\tilde{V}^{\bullet}\cO_X$ is decreasing by Remark \ref{remark: microlocal V filtration is decreasing}, we have
\[ \d_t(\gr^F_{j-1}V^1\cB_f)\supseteq \gr^F_{j}V^{>0}\cB_f,\]
and therefore
\[ \tilde{V}^{j}\cO_X/\tilde{V}^{>j}\cO_X\cong \frac{\tilde{V}^{\gamma}\cO_X\otimes \d_t^j}{\tilde{V}^{>\gamma}\cO_X\otimes \d_t^j}=\frac{\d_t(\gr^F_{j-1}V^1\cB_f)}{\gr^F_{j}V^{>0}\cB_f}.\]
Now the point is that one can relate the right hand side to $\gr^F_j\gr^0_V\cB_f$ using the compatibility between the Hodge filtration and $V$-filtration for Hodge modules. First, we always have
\[ \d_t: \gr^F_{j-1}V^1\cB_f \to \gr^F_jV^0\cB_f.\]
Since the Hodge module $\cB_f$ has strict support, by \cite[5.1.4]{Saito88} we have a surjection
\[
\d_t: F_{j-1}\gr^1_V\cB_f \twoheadrightarrow F_j\gr^0_V\cB_f.
\]
Passing to quotient preserves the surjectivity:
\begin{equation*}\label{eqn: strict support condition}
\d_t:\frac{ \gr^F_{j-1}V^1\cB_f}{\gr^F_{j-1}V^{>1}\cB_f }  =\gr^F_{j-1}\gr^1_V\cB_f\twoheadrightarrow \gr^F_j\gr^0_V\cB_f=\frac{\gr^F_jV^0\cB_f}{\gr^F_jV^{>0}\cB_f}.
\end{equation*}
We claim that the inclusion map 
\[ \frac{\d_t\cdot(\gr^F_{j-1}V^1\cB_f)}{\gr^F_jV^{>0}\cB_f}\to \frac{\gr^F_jV^0\cB_f}{\gr^F_jV^{>0}\cB_f}=\gr^F_j\gr^0_V\cB_f,\]
is an isomorphism, which will finish the proof. This claim is obtained by the following simple linear algebra lemma applied to 
\[ A=\gr^F_{j-1}V^1\cB_f, \quad B= \gr^F_jV^0\cB_f, \quad f=\d_t,\]
and
\[ A_0= \gr^F_{j-1}V^{>1}\cB_f, \quad B_0=\gr^F_jV^{>0}\cB_f.\]
\begin{lemma}
Let $f\colon A\to B$ be a linear map between two finite dimensional vector spaces. Suppose $A_0\subseteq A$ and $B_0\subseteq B$ are two subspaces respectively, so that $f(A_0)\subseteq B_0$ and the induced quotient map
\[ f: A/A_0 \to B/B_0\]
is surjective. Moreover assume that $f(A)\supseteq B_0$. Then the map 
\[ f(A)/B_0 \to B/B_0\]
induced by inclusion is surjective and hence an isomorphism.
\end{lemma}

\end{proof}

\begin{corollary}\label{corollary: minimal exponent via V filtration}
With the above notations, for $\gamma\in \Q$, we have $\tilde{\alpha}_f\geq \gamma$, or equivalently $\tilde{V}^{\gamma}\cO_X=\cO_X$,
if and only if the following holds
\begin{align*}
\gr^F_j\gr^{\alpha}_V\cB_f=0, \quad &\textrm{ for any $0 \leq j  \leq \lfloor \gamma \rfloor-1$ and $0\leq \alpha <1$}, \textrm{ and $j=
\lfloor\gamma \rfloor$, $0\leq \alpha <\gamma-\lfloor \gamma \rfloor$}.
 \end{align*}

\end{corollary}

\begin{proof}
Since $\tilde{V}^{\bullet}\cO_X$ is a decreasing filtration by Remark \ref{remark: microlocal V filtration is decreasing},  we have
\[ \tilde{V}^{\gamma}\cO_X=\cO_X\Longleftrightarrow \frac{\tilde{V}^{\beta}\cO_X}{\tilde{V}^{>\beta}\cO_X}=0, \quad \textrm{ for all $\beta <\gamma$}.\]
By Proposition \ref{proposition: jumping of microlocal V filtration}, we know that the right hand side equivalent to 
\[ \gr^F_{\lfloor \beta \rfloor } \gr^{\beta-\lfloor \beta \rfloor}_V\cB_f=0, \quad \textrm{ for all $\beta <\gamma$}. \]
This means that
\[ \gr^F_{j } \gr^{\alpha}_V\cB_f=0, \quad \textrm{ for all $j\in \Z$ and $0\leq \alpha <1$ such that $j+\alpha <\gamma$}, \]
and this gives what we want.
\end{proof}

%Now we can deduce Theorem \ref{thm: main 4}, which is a key result for the proof of   Theorems \ref{thm: main 1}-\ref{thm: main 2}.

%\begin{proof}
\subsection{Proof of Theorem \ref{thm: main 4}}
By Theorem \ref{thm: k du Bois singularity function via minimal exponent}, $Z$ has $k$-Du Bois singularities if and only if $\tilde{\alpha}_f\geq k+1$, which by Theorem \ref{thm: minimal exponent via microlocal V filtration} is equivalent to $\tilde{V}^{k+1}\cO_X=\cO_X$. By Corollary \ref{corollary: minimal exponent via V filtration}, this is equivalent to
\[  \gr^F_p\gr^{\alpha}_V\cB_f=0, \textrm{ for all $p\leq k$, $0\leq \alpha <1$},\]
which by Convention \ref{convention for V filtration} means that
\begin{equation*}
\gr^F_p\varphi^H_{f,\neq 1}\Q^H_X[\dim X]=0, \quad  p\leq k+1, \ \textrm{ and } \ \gr^F_{p}\varphi^H_{f,1}\Q^H_X[\dim X]=0, \quad p\leq k.
\end{equation*}

By Theorem \ref{thm: k du Bois singularity function via minimal exponent}, $Z$ has $k$-rational singularities if and only if $\tilde{\alpha}_f>k+1$, or equivalently $\tilde{V}^{>k+1}\cO_X=\cO_X$. In this case there is an additional vanishing
\begin{equation*} \tilde{V}^{k+1}\cO_X/\tilde{V}^{>k+1}\cO_X=0,\end{equation*}
which by Corollary \ref{corollary: minimal exponent via V filtration} and Convention \ref{convention for V filtration} is equivalent to 
\[ \gr^F_{k+1}\varphi^H_{f,1}\Q^H_X[\dim X]=\gr^F_{k+1}\gr^0_V\cB_f=0.\]
\qed
%\end{proof}

%%%%%%%%%%%%%%%%%%%%%%%%%%%%%%%%%%%

\subsection{Thom-Sebastiani type results}
In this section, we obtain a Thom-Sebastiani type statement for higher rational and higher Du Bois singularities, which relies on the Thom-Sebastiani theorem for microlocal $V$-filtration proved in \cite{MSS20}; compare also with \cite[Example 6.7]{MPQdivisor} for a local Thom-Sebastiani type statement.

For $a=1,2$, let $X_a$ be a smooth complex algebraic variety and $f_a$ be a non-constant holomorphic function on $X_a$ with $Z_a\colonequals f_a^{-1}(0)$. Set
\[ X\colonequals X_1\times X_2, \quad Z\colonequals f^{-1}(0)\subseteq X, \quad \textrm{with $f=f_1+f_2$ on $X$},\]
where $f_1+f_2$ is the abbreviation for $\mathrm{pr}^{\ast}_1f_1+\mathrm{pr}^{\ast}_2f_2$, with $\mathrm{pr}_a\colon X \to X_a$ the $a$-th projection. %Denote by $\tilde{V}^{\bullet}\cO_{X_a}$ the microlocal $V$-filtration on $\cO_{X_a}$ induced by $f_a$ as in Definition \ref{definition: microlocal V filtration}. %For $\cO_{X_a}$-modules $\cG_a$ on $X_a$, we set
%\[ \cG_1\boxtimes \cG_2 \colonequals \mathrm{pr}^{\ast}_1 \cG_1 \otimes_{\cO_X} \mathrm{\pr}^{\ast}_2\cG_2.\]
%In particular, $\cO_X=\cO_{X_1}\boxtimes \cO_{X_2}$.

%\begin{remark}\label{remark: change of convention of microlocal V filtration}
%    Note that the convention of the Hodge filtration on $\cB_f$ in \cite{MSS20} is different from ours, with a shift by the dimension of the ambient complex manifold. But if one compares \cite[(1.1.4)]{MSS20} and the equation above \cite[(2.1.6)]{MSS20} with our \eqref{eqn: associated graded of Hodge filtrations on Bf} and \eqref{eqn: OX embeds into the graph embedding}, the microlocal $V$-filtration in \cite{MSS20} is actually the same as ours. 
%\end{remark}

\begin{proposition}\label{TS}
    With the above notation, if $Z_1$ has $k_1$-Du Bois singularities and $Z_2$ has $k_2$-Du Bois singularities, then $Z$ has $(k_1+k_2+1)$-Du Bois singularities, by replacing $X_a$ with an open neighborhood of $Z_a$ in $X_a$ so that $\Sing(Z)=\Sing(Z_1)\times \Sing(Z_2)$ if necessary. 

Similarly, if $Z_1$ has $k_1$-rational singularities and $Z_2$ has $k_2$-du Bois singularities (or $Z_1$ is $k_1$-du Bois and $Z_2$ is $k_2$-rational), then $Z$ has $(k_1+k_2+1)$-rational singularities.
%In particular, if $Z_i$ has Du Bois singularities for $i=1,2$ then $Z$ has $1$-Du Bois singularities.
\end{proposition}

\begin{proof}
    First, by using the Thom-Sebastiani theorem for the microlocal $V$-filtration from \cite[(2.2.3)]{MSS20}, one gets the following implication for the minimal exponents
    \[ \tilde{\alpha}_{f_1}\geq k_1+1, \quad  \tilde{\alpha}_{f_2}\geq k_2+1 \Longrightarrow \tilde{\alpha}_f\geq k_1+k_2+2.\]
    Then the assertion of the Proposition follows from Theorem \ref{thm: k du Bois singularity function via minimal exponent} and Theorem \ref{thm: minimal exponent via microlocal V filtration}. The other assertion follows similarly. 
\end{proof}

\begin{remark}
Alternatively, one can prove Proposition \ref{TS} by using Theorem \ref{thm: main 4} and the Thom-Sebastiani formula for vanishing cycles proved in \cite[Theorem 2]{MSS20}. 
\end{remark}

%%%%%%%%%%%%%%%%%%%%%%%%%%%%%%%%%%%

\section{Higher multiplier ideals and applications}
In this section, we give an interpretation of higher Du Bois and higher rational singularities in terms of the higher multiplier ideals from \cite{SchnellYang}, and we prove Theorem \ref{prop: higher du Bois and higher rational via higher multiplier ideals} and Proposition \ref{thm: main 5}.

For any $k\in \N$ and $\beta \in \Q$, $\gr^F_kV^{\beta}\cB_f$ is a coherent subsheaf of $\gr^F_k\cB_f=\cO_X\otimes \d^k_t$ by \eqref{eqn: associated graded of Hodge filtrations on Bf}. Therefore it is natural to study $\gr^F_kV^{\beta}\cB_f$ as an ideal sheaf of $\cO_X$. This is considered both from local and global perspectives in \cite{SchnellYang}. For compatibility with the conventions for multiplier ideals, we use the following definition in this paper. 
\begin{definition}\label{definition: higher multiplier ideals}
    Let $X$ be a complex manifold, $f\colon X\to \C$ be a non-zero holomorphic function on $X$ such that $Z\colonequals f^{-1}(0)$ is reduced. For any $k\in \N$ and $\alpha \in \Q$, we define the \emph{$k$-th multiplier ideal} of $Z$ by
    \[\cJ_k(\alpha Z)\otimes \d_t^k\colonequals \gr^F_kV^{>\alpha}\cB_f\subseteq \gr^F_k\cB_f=\cO_X\otimes \d_t^k. \] 
\end{definition}
\begin{remark}
    In the notations of \cite{SchnellYang}, one has $\cJ_k(\alpha Z)=\cI_{k,<-\alpha}(Z)$, as in the introduction. By a  result of Budur-Saito from \cite{BS05} one gets that  $\cJ_0(\alpha Z)=\cJ(\alpha Z)$, the usual multiplier ideals. In this paper, we work with reduced hypersurfaces, but the construction in \cite{SchnellYang} actually applies without this assumption. 
    
    There is another higher version of multiplier ideals, called \emph{Hodge ideals}, studied extensively in recent years by Musta\c{t}\u{a} and Popa, e.g., see \cite{MP16,MP19, PopaICM}. The relation between the notion $\cJ_k(\alpha Z)$ and Hodge ideals is also considered in \cite{SchnellYang}. 
\end{remark}
It follows from Definition \ref{definition: higher multiplier ideals} that
\begin{equation}\label{eqn: associated graded of higher multiplier ideals}
\gr^F_k\gr^{\alpha}_V\cB_f\cong \cJ_k((\alpha-\epsilon)Z)/\cJ_k(\alpha Z).
\end{equation}
We can make the following:
\begin{definition}
  For any $k\in \N$, $\alpha \in \Q$ is a \emph{jumping number of the $k$-th multiplier ideal $\cJ_k$ of $Z$} if
  \[ \cJ_k((\alpha-\epsilon)Z)/\cJ_k(\alpha Z) \neq 0,\]
  or equivalently $\gr^F_k\gr^{\alpha}_V\cB_f\neq 0$.
\end{definition}

In view of Corollary \ref{corollary: minimal exponent via V filtration}, Theorem \ref{thm: minimal exponent via microlocal V filtration} can be rephrased as

\begin{proposition}\label{prop: minimal exponent via jumping numbers of higher multiplier ideals}
    Let $X$ be a smooth complex algebraic variety and $f\colon X\to 
    \C$ a non-zero holomorphic function such that $Z\colonequals f^{-1}(0)$. Then 
\begin{align*}
\tilde{\alpha}_f
&=\min\{k+\alpha, k\in \Z, \alpha\in[0,1) \mid \cJ_k((\alpha-\epsilon)Z)/\cJ_k(\alpha Z)\neq 0\}.
\end{align*}
 \end{proposition}

\begin{proof}
    The assertion follows by combining Theorem \ref{thm: minimal exponent via microlocal V filtration}, Corollary \ref{corollary: minimal exponent via V filtration} and \eqref{eqn: associated graded of higher multiplier ideals}.
\end{proof}
 
%Now we can prove Proposition \ref{prop: higher du Bois and higher rational via higher multiplier ideals}, which reinterprets Theorem \ref{thm: main 4} in terms of higher multiplier ideals. % and is also used for the proof of Theorem \ref{thm: main 3}.

%\begin{proof}
\subsection{Proof of Theorem \ref{prop: higher du Bois and higher rational via higher multiplier ideals}}
    This follows by combining Theorem \ref{thm: k du Bois singularity function via minimal exponent} and Proposition \ref{prop: minimal exponent via jumping numbers of higher multiplier ideals}. \qed
%\end{proof}

\subsection{Proof of Proposition \ref{thm: main 5}}
If $\alpha \in [0,1)\cap \Q$ is not a jumping number of the $\ell$-th multiplier ideal $\cJ_\ell$ for $Z$ for all $\ell\leq k$, then
\[ \cJ_{\ell}((\alpha-\epsilon)Z)=\cJ_{\ell}(\alpha Z), \quad \textrm{for $0<\epsilon \ll 1$ and all $ \ell \leq k$}.\]
In view of \eqref{eqn: associated graded of higher multiplier ideals}, this is equivalent to
\[ \gr^F_\ell\gr^{\alpha}_V\cB_f=0, \quad \textrm{for all $\ell \leq k$}.\]
Since $\gr^\alpha_V \cB_f$ is the \emph{left} $\sD$-module underlying $\varphi^H_{f,e^{-2\pi i\alpha}}\Q^H_X[\dim X]$, by Convention \ref{convention for V filtration} the above vanishing translates into 
\[\gr^F_\ell \varphi^H_{f,e^{-2\pi i\alpha}}\Q^H_X[\dim X] =0,\quad \textrm{for all $\ell \leq k+1$}, \quad \textrm{if $\alpha \in (0,1) \cap \Q$} , \] and
\[\gr^F_\ell \varphi^H_{f,1}\Q^H_X[\dim X] =0,\quad \textrm{for all $\ell \leq k$}, \quad \textrm{if $\alpha=0$.} \]
\qed
%\end{proof}

\iffalse
\begin{remark}

    Using results from \cite{SchnellYang}, it can be shown further that $Z$ having $k$-Du Bois singularities is equivalent to
    \[ \cJ_{\ell}((1-\epsilon)Z)=\cO_X, \quad \textrm{for $0<\epsilon \ll 1$ and all $0\leq \ell \leq k$},\]
%    or equivalently
%    \[ \cI_{\ell,-1}(Z)=\cO_X,\quad \textrm{for all $0\leq \ell \leq k$}.\]
    which can be seen as the higher multiplier ideal version of \cite[Theorem 1]{JKSY}.
    One can also show that $Z$ having $k$-rational singularities is equivalent to  $k$-Du Bois and
    \[ \cJ_{k+1}(-\epsilon Z)=\cO_X, \quad \textrm{for $0<\epsilon \ll 1$}.\]
     For details, see \cite{SchnellYang}.
\end{remark}
\fi

%%%%%%%%%%%%%%%%%%%%%%%%%%%%%%%%%%%%%%%%%%%%%%

\section{Higher Du Bois and higher rational singularities via characteristic classes}\label{sec:2} %{\color{red} PLEASE DO NOT ALTER SECTION 2 ANYMORE!}
In this section we review basic definitions and results concerning Hirzebruch classes of mixed Hodge modules  introduced by Brasselet-Sch\"urmann-Yokura in \cite{BSY} (see also \cite{S, MS}), as well as the  (spectral) Hirzebruch-Milnor classes of hypersurfaces \cite{CMSS, MSS1, MSS}. We also prove Theorem \ref{dbrsp} after some preparatory statements.

Let $Z$ be a complex algebraic variety, and let $\mhm(Z)$ be the abelian category of mixed Hodge modules on $Z$. We let $H_k(Z)$ denote either the Borel-Moore homology group $H^{BM}_{2k}(Z,\mathbb{Q})$ or the Chow group $CH_k(Z)_\mathbb{Q}$,  and we let $K_0(Z)$ denote the Grothendieck group of coherent sheaves on $Z$.

\subsection{Hirzebruch classes of mixed Hodge modules and Hirzebruch-Milnor classes of hypersurfaces}\label{Hirzebruch classes of mixed Hodge modules and Hirzebruch-Milnor classes of hypersurfaces}
In this section, we review the construction of the {Hirzebruch class transformation}. First, we introduce the motivic Chern class transformation
\[
 \DR_y: K_0(\mhm(Z)) \to  K_0(Z)[y,y^{-1}],
\]
which is defined as follows. If $Z$ is smooth and $M\in \mhm(Z)$, let $(\cM,F_{\bullet}\cM)$ be the underlying  filtered \emph{left} $\sD$-module of $M$, and set
\[
\DR_y[M]\colonequals \sum_{p}  \left[\gr^F_{-p}\DR(\cM)\right] \cdot(-y)^p \in K_0(Z)[y,y^{-1}].
\]
The definition extends to the case when $Z$ is singular by using locally defined closed embeddings into smooth varieties. Furthermore, it can also be extended to complexes $M^\bullet \in D^b\mhm(Z)$ by applying it to each cohomology module $H^iM^\bullet\in \mhm(Z)$, $i \in \Z$. 

\begin{definition}
The {\it Hirzebruch class transformation}  is defined by 
\begin{align*}
T_{y\ast}: K_0(\mhm(Z)) &\to H_*(Z)[y, y^{-1}] \\
[M^\bullet] &\mapsto \td_{\ast} (\DR_y[M^\bullet]),
\end{align*}
where $\td_{\ast}:K_0(Z) \to H_\ast(Z)$
is the 
Baum-Fulton-MacPherson Todd class transformation \cite{BFM}, which is linearly extended over $\mathbb{Z}[y,y^{-1}]$. 
\end{definition}

\begin{remark}\label{negy}
%A priori, we have that $T_{y\ast}(M^\bullet)\in H_\ast(Z)\big[y,\frac{1}{y(y+1)}\big]$, but  it was shown in \cite[Proposition 5.21]{S} that in fact 
%\[ T_{y\ast}(M^\bullet) \in H_{\ast}(Z)[y,y^{-1}].\]
The above definition corresponds to the {\it un-normalized} Hirzebruch class transformation of \cite{BSY}, as opposed to the normalized version $\widehat{T}_{y\ast}$ used in \cite{CMSS, MSS1, MS}, which requires a further twisting  of $\td_\ast$ by powers of $(1+y)$. When precomposed with the group homomorphism $\chi_{Hdg}:K_0(var/Z) \to K_0(\mhm(Z))$ given by $[f\colon Y \to Z] \mapsto [f_!\mathbb{Q}^H_Y]$, one gets motivic transformations ${T}_{y\ast}, \widehat{T}_{y\ast}: K_0(var/Z) \to H_*(Z)[y]$, see \cite[Theorem 3.1]{BSY}. 
%Hence, if no negative powers of $y$ appear in $T_{y\ast}([M^\bullet])$, e.g., when $[M^\bullet]$ is in the image of $\chi_{Hdg}$, then the evaluations at $y=0$ of $T_{y\ast}([M^\bullet])$ and $\widehat{T}_{y\ast}([M^\bullet])$ agree. %In particular, $T_{y\ast}(\mathbb{Q}^H_Z)\in H_*(Z)[y].$
\end{remark}
%Let $\Q^H_X$ be the constant Hodge module on $X$, whose underlying constructible complex is $\Q_X$. The (normalized) homology {\it Hirzebruch class of $X$} is defined as
%\[T_{y\ast}(X):=T_{y\ast}(\Q^H_X) \in H_\ast(X)[y],\]
%where the last inclusion is reduced to the case when $X$ is smooth, in which case it follows from its relation to the cohomology Hirzebruch class of the generalized Hirzebruch-Riemann-Roch. \textcolor{red}{Ruijie: I don't understand the discussion here, is it just for completeness? It seems that we actually don't use the normalized Hirzebruch class for X.}

\begin{remark}\label{remark: shifting change of sign}
Since the Hirzebruch class transformation is defined on the Grothendieck group of mixed Hodge modules, shifting a complex has the effect of changing its Hirzebruch class by a sign, i.e., $T_{y\ast}([M^\bullet[d]])=(-1)^d \cdot T_{y\ast}([M^\bullet]).$
\end{remark}

%\subsection{Hirzebruch-Milnor classes}
By applying the above Hirzebruch class transformation to a vanishing cycle complex, one gets (an un-normalized version of) the \emph{Hirzebruch-Milnor classes} introduced in \cite{CMSS,MSS1}. Let  
$X$ be a smooth complex algebraic variety and let $f\colon X\to \C$ be a non-constant holomorphic function. Set $Z\colonequals f^{-1}(0)$. Let $\varphi^H_f$ be the vanishing cycle functor of mixed Hodge modules. As in the introduction, we work with the shifted functor $\varphi_f:=\varphi^H_f[1]$, whose underlying functor on constructible complexes is the Deligne vanishing cycle functor.
\begin{definition}
The {\it Hirzebruch-Milnor class of $Z$} is defined by:
\begin{equation*}\label{HMc}
\begin{split}
    M_{y\ast}(Z)&:=T_{y\ast}(\varphi_f\Q^H_X)\\
    &=\td_\ast \left( \sum_p \left[\gr^F_{-p}\DR(\varphi_f\Q^H_X)\right] \cdot(-y)^p \right) \in H_*(\Sing(Z))[y],
    \end{split}
\end{equation*} 
where $\Q_X^H$ is the shifted constant Hodge module on $X$, whose underlying constructible complex is $\Q_X$.
\end{definition}

\begin{remark}\label{negym}
Using Sch\"urmann's specialization for $T_{y\ast}$, cf. \cite[eq.(8)]{Sp}, it can be shown as in \cite[Theorem 3.2]{CMSS} that $M_{y\ast}(Z)$ is the difference between the ``virtual'' and the ``actual'' un-normalized Hirzebruch classes of $Z$. Since this specialization fits with the motivic specialization (compatibly with the semi-simple part of the monodromy action as in \cite[(10)]{Sp} or \cite[(1.14)]{CMSS}), this also shows that no negative powers of $y$ appear in $M_{y\ast}(Z)$, as in Remark \ref{negy}.  Moreover, as our calculations show, one can see directly that $\DR_y[\varphi_f\Q^H_X]\in K_0(Z)[y]$. %, see \eqref{eqn: associated graded of Hodge filtrations on Bf} and Convention \ref{convention for V filtration}.
\end{remark}

%%%%%%%%%%%%%%%

\subsection{Spectral classes}\label{sec: spectral classes}
The Hirzebruch class transformation $T_{y\ast}$ has been lifted in \cite{MSS} to a spectral version, i.e., a characteristic class version of the Hodge spectrum. Here we work with an un-normalized version of this class.

As in the previous section, let us first assume that $Z$ is a smooth complex algebraic variety. Let $\mhm(Z,T_s)$ be the abelian category of mixed Hodge modules $M$ on $Z$ which are endowed with an action of $T_s$ of finite order. For $(M, T_s) \in \mhm(Z,T_s)$
with $T_s^e=\mathrm{Id}$, let $(\cM,F_{\bullet}\cM)$ be the underlying filtered \emph{left} $\sD_X$-module of $M$.  There is a canonical decomposition 
$$(\cM,F_{\bullet})=\sum_{\lambda\in \mu_e} (\cM_\lambda,F_{\bullet}),$$
such that $T_s=\lambda\cdot \mathrm{Id}$ on $\cM_\lambda$, where $\mu_e=\{\lambda \in \C \mid \lambda^e=1\}$. 
With these notations, %we have a spectral version of the motivic Chern class transformation, defined by
we define the spectral motivic Chern class of $(M, T_s)$ by
\begin{align*}
%\DR_t[M,T_s]: \mhm(Z,T_s) & \to K_0(Z)[t^{ 1/e},t^{-1/e}],\\
\DR_t[M, T_s] := \sum_{p,\lambda} [\gr^F_{-p}\DR(\cM_\lambda)] \cdot  t^{p+\ell(\lambda)} \in K_0(Z)[t^{ 1/e},t^{-1/e}],
\end{align*}
where $ \ell(\lambda)$ is the unique number in $[0,1)$ such that $e^{2\pi i \ell(\lambda)}=\lambda$. As in the case of the motivic Chern class transformation, the above definition extends to the case when $Z$ is singular and also to complexes $(M^{\bullet}, T_s)$ of mixed Hodge modules on $Z$ endowed with the action of a finite order automorphism $T_s$.  

Let $K^{mon}_0(\mhm(Z))$ denote the Grothendieck group of mixed Hodge modules on $Z$ endowed with a finite order automorphism. %$\mhm(Z,T_s)$.
\begin{definition}
The {\it spectral Hirzebruch transformation} is defined by
\[ T^{sp}_{t\ast}\colon  K^{mon}_0(\mhm(Z)) \to \bigcup_{e\geq 1} H_\ast(Z)\left[t^{1/e}, t^{-1/e}\right]\]
\[
[M^{\bullet},T_s]\mapsto \td_{\ast}\left(\DR_t[M^{\bullet}, T_s]\right)  ,
\]
where $\td_\ast$ is as before the Todd class transformation.
\end{definition}
%\begin{remark}
%In fact, by \cite[Proposition 1.4]{MSS} we have 
%\[ T^{sp}_{t\ast}(M^{\bullet},T_s) \in \bigcup_{e\geq 1} H_\ast(Z)[t^{1/e}, t^{-1/e}].\] 
%\end{remark}
%by using locally defined closed embeddings into smooth varieties, and it can be extended to complexes $M^\bullet \in D^b\mhm(X)$ endowed with an action of $T_s$ of finite order, by applying it to each cohomology module $H^i M^\bullet \in \MHM(X,T_s)$ ($i \in \Z)$.

 %Altogether, one gets a natural transformation on the the   of algebraic mixed Hodge modules with a finite order automorphism, called the {\it spectral Hirzebruch class transformation}.  

%Finally note that by letting $t=-y$ and forgetting the action, $T^{sp}_{t\ast}(\cM^\bullet,T_s)$ specializes to the usual Hirzebruch class $T_{y\ast}(\cM^\bullet)$. 

Let us now restrict to the case of globally defined hypersurfaces. Let $X$ be a smooth complex algebraic variety and let $f\colon X \to \C$ be a non-constant holomorphic function with $Z=f^{-1}(0)$. Let 
\[\varphi_f\Q^H_X \in \mhm(Z)[1-\dim(X)]\]
be the (shifted) vanishing cycle mixed Hodge module, and let $T_s$ be the semisimple part of the monodromy. Then 
one can introduce the following un-normalized version of the spectral Hirzebruch class from \cite{MSS} (where one can argue as in Remark \ref{negym} that no negative fractional powers of $y$ appear in its definition).
\begin{definition}\label{definition: spectral Hirzebruch-Milnor class}
The {\it localized spectral Hirzebruch-Milnor class of $Z$} is defined by:
\[ {\cM^{sp}_{t*}(Z):=T^{sp}_{t*}(\varphi_f\Q^H_X,T_s)} \in H_*(\Sing(Z))[t^{1/ord(T_s)}].\]
\end{definition}

In view of the above definitions, it follows immediately that formula \eqref{MHsp} holds, i.e., in the notations of the introduction, one has the following identity for any integer $k \geq 0$:
\begin{equation}\label{eqn: Hirzebruch-Milnor and spectral}
 [M_{(-y)\ast}(Z)]_{k}=\bigoplus_{\alpha \in \mathbb{Q}\cap[0,1)} \cM^{sp}_{t\ast}(Z)|_{t^{k+\alpha}}.
\end{equation}

%%%%%%%%%%%%%%%%%%%%%%%%%%%%%%%%%%%%%%%%%%%%%%$

\subsection{Preparatory results} In order to prove Theorem \ref{dbrsp}, we need a few prerequisites.

Let $(\cM,F_{\bullet}\cM)$ be a left filtered $\sD$-module on $X$. We denote $d_X:=\dim X$. The associated graded of the de Rham complex is given by
\[ \gr^F_\ell\DR(\cM)\colonequals [ \gr^F_{\ell}\cM\to \Omega^1_X\otimes \gr^F_{\ell+1}\cM \to \cdots %\to \Omega^{d_X-1}_X\otimes \gr^F_{\ell+d_X-1}\cM 
\to \Omega^{d_X}_X\otimes \gr^F_{\ell+d_X}\cM][d_X]. \]
\begin{lemma}\label{lemma: equivalence between vanishing of modules and de Rham complexes}
With the above notations, let $p\in \Z$ such that $F_{p-1}\cM=0$. For $k\geq 0$, the vanishing
\[ \gr^F_{p+j}\cM=0, \quad \forall 0\leq j\leq k, \]
is equivalent to
\[ \gr^F_{p+j}\DR(\cM)=0, \quad \forall -d_X\leq j\leq -d_X+k.\]
\end{lemma}

\begin{proof}
    We prove the assertion by induction on $k$. The base case is $k=0$, which follows from
    \begin{align*}
    \gr^F_{p-d_X}\DR(\cM)&=[0\to 0 \to \ldots \to \Omega^{d_X}_X\otimes \gr^F_{p}\cM][d_X]=\omega_X\otimes \gr^F_p\cM.
    \end{align*}
Suppose we know the statement for $k-1$. Assume
\[ \gr^F_{p+j}\DR(\cM)=0, \quad \forall -d_X\leq j\leq -d_X+k.\]
By induction, this is equivalent to
\begin{equation}\label{eqn: vanishing consequence of grFp+jDR}
F_{p+k-1}\cM=0 \quad \textrm{and } \ \gr^F_{p-d_X+k}\DR(\cM)=0.
\end{equation}
Therefore
   \begin{align*}
    \gr^F_{p-d_X+k}\DR(\cM)&=[0\to 0 \to \ldots \to \Omega^{d_X}_X\otimes \gr^F_{p+k}\cM][d_X]=\omega_X\otimes \gr^F_{p+k}\cM.
    \end{align*}
    Hence the vanishing \eqref{eqn: vanishing consequence of grFp+jDR} is equivalent to $\gr^F_{p+j}\cM=0$ for $0\leq j\leq k$ and this finishes the proof.
\end{proof}

The next lemma concerns the duality for the vanishing cycle mixed Hodge modules. %It matches up with duality for left filtered $\sD$-modules underlying pure Hodge modules, see \cite[Lemma 8.4]{Schnellvanishing}.
\begin{lemma}\label{lemma: duality function for left mixed Hodge module}
Let $\mathbb{D}$ denote the duality functor for mixed Hodge modules. Denote $d=\dim X-1=\dim Z$. Then for any $j\geq 0$, we have isomorphisms for the associated graded of de Rham complexes of the underlying left $\sD$-modules:
\begin{align}\label{eqn: duality non unipotent left}
\D\left(\gr^F_j\DR(\varphi_{f,\neq 1}\Q^H_X[d])\right)&=\gr^F_{-j-d}\DR(\varphi_{f,\neq 1}\Q^H_X[d]),
\end{align}
\begin{align}\label{eqn: duality unipotent left}
\D\left(\gr^F_j\DR(\varphi_{f,1}\Q^H_X[d])\right)&=\gr^F_{-j-1-d}\DR(\varphi_{f,1}\Q^H_X[d]).
\end{align}
\end{lemma}
\begin{proof}
By \cite[(4.5.1)]{MSS} or \cite[Section 5.2]{Saito88}, there are self-dualities
\begin{equation}\label{selfd1}
\D(\varphi_{f,\neq 1}\Q^H_X[d])=(\varphi_{f,\neq 1}\Q^H_X[d])(d),
\end{equation}
\begin{equation}\label{selfd2}
\D(\varphi_{f,1}\Q^H_X[d])=(\varphi_{f,1}\Q^H_X[d])(d+1).
\end{equation}
This gives the isomorphisms of associated graded pieces of Hodge filtrations on de Rham complexes of the underlying \emph{right} $\sD$-modules:
\begin{align}\label{eqn: duality for non-unipotent vanishing cycle}
\D(\gr^j_F\DR(\varphi_{f,\neq 1}\Q^H_X[d]))&=\gr^{-j+d}_F\DR(\varphi_{f,\neq 1}\Q^H_X[d]),
\end{align}
\begin{align}\label{eqn: duality for unipotent vanishing cycle}
\D(\gr^j_F\DR(\varphi_{f,1}\Q^H_X[d]))&=\gr^{-j+d+1}_F\DR(\varphi_{f,1}\Q^H_X[d]).
\end{align}
Recall the conventions between decreasing and increasing filtrations for both left and right $\sD$-modules:
\[ F^{j}\cM=F_{-j}\cM \Longrightarrow \gr^F_j\DR(\cM)=\gr_F^{-j}\DR(\cM).\]
By \cite[(1.2.2)]{MSS}, if $(\cM^{\ell},F_{\bullet}\cM^{\ell})$ is a \emph{left} $\sD$-module on a complex manifold $X$, with associated \emph{right} $\sD$-module $(\cM^r,F_{\bullet}\cM^r)=(\omega_X,F_{\bullet})\otimes_{\cO_X}(\cM^{\ell},F_{\bullet})$ where $\gr^F_{\ell}\omega_X=0$ if $\ell\neq -\dim X$, there is a canonical filtered isomorphism
\begin{align}\label{eqn: DR}
 \DR(\cM^{\ell},F_{\bullet})=\DR(\cM^{r},F_{\bullet}).
\end{align}
This is because even though there is a shift of filtration on $\sD$-modules for the left-to-right transformation, for a \emph{right} filtered $\sD$-module $(\cM^r,F_{\bullet}\cM^r)$ on $X$ we have
\[ \gr^F_p\DR(\cM^r)=[\gr^F_{p-\dim X}\cM^r\otimes \wedge^{\dim X} \sT_X \to \ldots \to \gr^F_{p}\cM^r][\dim X],\]
where $\cT_X$ is the tangent bundle of $X$.

Let us now prove \eqref{eqn: duality non unipotent left} using \eqref{eqn: duality for non-unipotent vanishing cycle}.  First, using \eqref{eqn: DR}, we may work directly with the underlying right $\sD$-modules, in which case we have
\begin{align*}
\D\left(\gr^F_j\DR(\varphi_{f,\neq 1}\Q^H_X[d])\right)
%&=\D\left(\gr^F_j\DR(\varphi_{f,\neq 1}\Q^H_X[d])\right) \quad \textrm{from left to right $\sD$-modules}\\
&=\D\left(\gr^{-j}_F\DR(\varphi_{f,\neq 1}\Q^H_X[d])\right)\\
&\overset{\eqref{eqn: duality for non-unipotent vanishing cycle}}{=}\gr^{j+d}_F\DR(\varphi_{f,\neq 1}\Q^H_X[d])=\gr^F_{-j-d}\DR(\varphi_{f,\neq 1}\Q^H_X[d]).
%&=\gr^F_{-j-d}\DR(\varphi_{f,\neq 1}\Q^H_X[d]) \quad \textrm{ from right to left $\sD$-modules}.
\end{align*}

The equality \eqref{eqn: duality unipotent left} is proved similarly using \eqref{eqn: duality for unipotent vanishing cycle}.
\end{proof}

\begin{remark}\label{deig}
    The duality isomorphisms \eqref{selfd1} and \eqref{selfd2} are compatible with the action of the semisimple part of the monodromy $T_s$, where the $e^{2\pi i \alpha}$-eigenspace is the dual of the $e^{-2\pi i \alpha}$-eigenspace, see \cite[(2.4.3)]{Saito94}.
\end{remark}

\subsection{Proof of Theorem \ref{dbrsp}}
As in Lemma \ref{lemma: duality function for left mixed Hodge module}, we let $d=\dim X-1=\dim Z$. Suppose $Z$ has $k$-Du Bois singularities. This is equivalent by Theorem \ref{prop: higher du Bois and higher rational via higher multiplier ideals} and Proposition \ref{thm: main 5} to the vanishing
\[\gr^F_p \varphi_{f,e^{-2\pi i\alpha}}\Q^H_X[d] =0,\quad \textrm{for all $p \leq k+1$}, \quad \textrm{if $\alpha \in (0,1) \cap \Q$} , \] and
\[\gr^F_p \varphi_{f,1}\Q^H_X[d] =0,\quad \textrm{for all $ p \leq k$} \quad (\textrm{if $\alpha=0$).} \]
By Lemma \ref{lemma: equivalence between vanishing of modules and de Rham complexes}, the above two vanishing conditions are equivalent to
\begin{align}\label{eqn: vanishing of DR up to k}
\gr^F_p\DR(\varphi_{f,e^{-2\pi i\alpha}}\Q^H_X[d])=0, &\quad \forall  p\leq -d+k, \quad \textrm{if $\alpha \in (0,1) \cap \Q$} \nonumber , \\
\gr^F_{p}\DR(\varphi_{f,1}\Q^H_X[d])=0, &\quad \forall p\leq -d+k-1 \quad (\textrm{if $\alpha=0$).}
\end{align}
Applying the duality functor $\D$ and using Remark \ref{deig} and Lemma \ref{lemma: duality function for left mixed Hodge module}, this is further equivalent to
\begin{equation}\label{eqn: vanishing of first k terms in De Rham} \gr^F_{-p}\DR(\varphi_{f,e^{2\pi i\alpha}}\Q^H_X[d])=0,\quad \textrm{for all $p \leq k$} \quad \textrm{$\alpha \in [0,1)\cap \Q$}.\end{equation}
By the definition of spectral classes in \S \ref{sec: spectral classes} and Remark \ref{remark: shifting change of sign}, this implies that
\[ {\cM^{sp}_{t*}(Z)}|_{t^{p+\alpha}}=\mathrm{td}_{\ast}\left(\left[\gr^F_{-p}\DR(\varphi_{f,e^{2\pi i \alpha}}\Q^H_X)\right]\right)=0, \quad \textrm{for all \ $0\leq p \leq k, \quad \alpha\in[0,1)\cap \Q$}.\]

Assuming $\Sing(Z)$ is projective, let us prove that the converse holds. If $Z$ does not have $k$-Du Bois singularities, then the vanishing \eqref{eqn: vanishing of first k terms in De Rham} cannot hold, so there must exists some integer $p$ with $0\leq p \leq k$ such that
\[ \gr^F_{-p}\DR(\varphi_{f}\Q^H_X[d])\neq 0.\]
The projectivity of $\Sing(Z)$ together with the positivity property of the Todd class transformation $\td_*$ on a projective space (e.g., see \cite[Section 1.3]{MSS}) then imply that 
\[ [M_{y\ast}(Z)]_{p}\neq 0 \textrm{ in } H_{\bullet}(\Sing Z),\]
contradicting the hypothesis. Indeed, here we use the fact that if the support of a coherent sheaf $\cF$ on a projective variety $Y$ has support of dimension $j$, then the degree $p$ part of $\td_*([\cF])$ vanishes for $p>2j$ and its degree $2j$ part is positive. This assertion can be reduced by the functoriality of $\td_*$ to the case when $Y$ is the projective space, where it follows from the positivity of the degrees of subvarieties.

Suppose now that $Z$ has $k$-rational singularities. By Theorem \ref{thm: main 4}, this is equivalent to $Z$ having $k$-Du Bois singularities with an additional vanishing
\begin{equation}\label{eqn: additional vanishing for rational singularity}  
\gr^F_{k+1}\varphi_{f,1}\Q^H_X[d]=0. 
\end{equation}
Using \eqref{eqn: vanishing of DR up to k} and Lemma \ref{lemma: equivalence between vanishing of modules and de Rham complexes}, we have
\[ \gr^F_{k-d}\DR(\varphi_{f,1}\Q^H_X[d])=0.\]
Using Lemma \ref{lemma: duality function for left mixed Hodge module}, the latter vanishing  is equivalent by duality to
\[
\gr^F_{-k-1}\DR(\varphi_{f,1}\Q^H_X[d])=0.
\]
Applying the Todd class transformation $\td_*$, this gives
\[ \cM^{sp}_{t\ast}(Z)\vert_{t^{k+1}}=0.\]
If $\Sing(Z)$ is projective, the converse holds using the same argument as above via the positivity of Todd classes. \qed

\subsection{Concluding remarks}
\begin{remark}\label{remnew}
We note here that $k$-rationality can also be detected by the un-normalized Hirzebruch-Milnor classes without using the semisimple part of monodromy. Indeed, if $Z$ is $k$-rational, this is equivalent by 
    Theorem  \ref{thm: main 4} to the condition 
    \begin{align*}
\gr^F_p\varphi_{f}\Q^H_X[d]=0, &\quad \forall \ p\leq k+1.
\end{align*}
By Lemma \ref{lemma: equivalence between vanishing of modules and de Rham complexes}, this is equivalent to
\begin{align*}
\gr^F_p\DR(\varphi_{f}\Q^H_X[d])=0, &\quad \forall  p\leq -d+k\nonumber ,
\end{align*}
and hence
\[ [M_{y\ast}(Z)]_{p}=\mathrm{td}_{\ast}\left(\left[\gr^F_{-p}\DR(\varphi_{f}\Q^H_X)\right]\right)=0  \quad \forall p\geq d- k,\]
with the converse true if $\Sing(Z)$ is projective.
\end{remark}

\begin{remark}
When $\Sing(Z)$ is projective, 
one can supply a purely homological proof of Proposition \ref{TS}, based on our Theorem \ref{dbrsp}, together with a corresponding Thom-Sebastiani formula for the spectral Milnor-Hirzebruch classes proved in \cite[Theorem 4]{MSS}. 
\end{remark}
%%%%%%%%%%%%%%%%%%%%%%%%%%%%%%%%%%%%%%%%%%%%%%

\end{document}